\documentclass[11pt]{article}
\usepackage{amsmath}
\usepackage[thmmarks,amsmath,amsthm]{ntheorem}
\usepackage{amsfonts, psfrag, graphicx, amssymb,enumitem,indentfirst,float,geometry,color,enumerate,graphicx,subfigure,algorithm,algpseudocode,pifont,hyperref}
\usepackage{epstopdf}
\usepackage{stmaryrd}

\allowdisplaybreaks[4]
\numberwithin{equation}{section}
\numberwithin{figure}{section}

\usepackage{lipsum}
\newcommand\blfootnote[1]{%
  \begingroup
  \renewcommand\thefootnote{}\footnote{#1}%
  \addtocounter{footnote}{-1}%
  \endgroup
}

\newtheorem{asp}{Assumpsion}[section]

\newtheorem{thm}{Theorem}[section]
\newtheorem{lemma}{Lemma}[section]
\newtheorem{rem}{Remark}[section]

\newcommand{\commentout}[1]{{}} 

\newcommand{\abs}[1]{\left|#1\right|}

\newcommand{\vertiii}[1]{{\left\vert\kern-0.25ex\left\vert\kern-0.25ex\left\vert #1
    \right\vert\kern-0.25ex\right\vert\kern-0.25ex\right\vert}}

\begin{document}
\title{Error Analysis Of Symmetric Linear/Bilinear Partially Penalized Immersed Finite Element Methods For Helmholtz Interface Problems \\
}
\author{ Ruchi Guo \thanks  {Department of Mathematics, Virginia Tech, Blacksburg, VA 24061 (ruchi91@vt.edu)}
\and
Tao Lin \thanks  {Department of Mathematics, Virginia Tech, Blacksburg, VA 24061 (tlin@vt.edu)}
\and
Yanping Lin \thanks{Department of Applied Mathematics, Hong Kong Polytechnic University, Hong Kong, China (yanping.lin@polyu.edu.hk)}
\and
Qiao Zhuang \thanks{Department of Mathematics, Virginia Tech, Blacksburg, VA 24061 (Corresponding author) (qzhuang@vt.edu)}}
\date{}
\maketitle
\begin{abstract}
This article presents an error analysis of the symmetric linear/bilinear partially penalized immersed finite element (PPIFE) methods for interface problems of Helmholtz equations. Under the assumption that the exact solution possesses a usual piecewise $H^2$ regularity, the optimal error bounds for the PPIFE solutions are derived in an energy norm and the usual $L^2$ norm provided that the mesh size is sufficiently small. A numerical example is conducted to validate the theoretical conclusions.
\end{abstract}
\section{Introduction}

This article is about the error analysis for the linear and bilinear partially penalized immersed finite element (PPIFE) methods developed in
\cite{2019LinLinZhuang} for solving  interface boundary value problems of the Helmholtz equation \cite{1984Brown, 2006KreissPetersson}
that is posed in a bounded domain $\Omega\subseteq\mathbb{R}^2$: find $u(X)$ that satisfies the Helmholtz equation and the first-order absorbing boundary condition:
\begin{subequations}\label{model}
\begin{align}
\label{inter_PDE}
 -\nabla\cdot(\beta\nabla u)-k^2u=f, \;\;\;\; & \text{in} \; \Omega^-  \cup \Omega^+, \\
 \label{bnd_cond}
 \beta \frac{\partial{u}}{\partial{\boldsymbol{n}_{\Omega}}}+ik u=g, \;\;\;\; &\text{on} \; \partial\Omega,
\end{align}
together with the jump conditions across the interface
\cite{2007KlenowNisewongerBatraBrown,1984Brown, 1988ChristiansenKrenk, 2011JensenKupermanSchmidt,2006KreissPetersson}:
\begin{align}
[u]_{\Gamma} &:=u^-|_{\Gamma} - u^+|_{\Gamma}= 0, \label{jump_cond_1} \\
\big[\beta \nabla u\cdot \mathbf{n}\big]_{\Gamma} &:= \beta^- \nabla u^-\cdot \mathbf{n}|_{\Gamma} - \beta^+ \nabla u^+\cdot \mathbf{n}|_{\Gamma} = 0, \label{jump_cond_2}
\end{align}
where the domain $\Omega\subseteq\mathbb{R}^2$ is divided by an interface curve $\Gamma$ into two subdomains $\Omega^-$ and $\Omega^+$,  with $\overline{\Omega}=\overline{\Omega^-\cup {\Omega^+}\cup\Gamma}$ ,  $u^s = u|_{\Omega^s}, s = \pm$ and $\mathbf{ n}$ is the unit normal vector to the interface $\Gamma$, $k$ is the wave number, $i=\sqrt{-1}$,  $\mathbf{n}_{\Omega}$ is the unit outward normal vector to $\partial \Omega$, and the coefficient $\beta$ is a piecewise positive constant function such that
\begin{equation*}
\beta(X)=
\left\{\begin{array}{cc}
\beta^- & \text{for} \; X\in \Omega^- ,\\
\beta^+ & \text{for} \; X\in \Omega^+.
\end{array}\right.
\end{equation*}
\end{subequations}
\commentout{
The phenomena of wave propagation in heterogeneous media appear in many aspects of science and engineering and have been widely studied, such as those investigations about wave transmission, diffraction, and scattering  ~\cite{2007KlenowNisewongerBatraBrown,1999WildeZhang, 1988ChristiansenKrenk,2000DhiaCiarletZwolf,1989Speck,2016ZhangLi},
to name just a few. Helmholtz interface problems arise in the study on time-harmonic wave propagation in heterogeneous media~\cite{ 2015Chaumont_thesis, 2016BarucpChaumontGout}. Across the interface between two different materials, the amplitude needs to satisfy the jump conditions~\cite{ 2007KlenowNisewongerBatraBrown,1984Brown, 1988ChristiansenKrenk, 2011JensenKupermanSchmidt, 2006KreissPetersson, 2017ZouWilkinsHarari} or compatibility conditions~\cite{ 2015Chaumont_thesis, 2016BarucpChaumontGout}.
The jump conditions are imposed on the basis of pertinent physical properties, for instance, the continuity of pressure and volume flow \cite{2007KlenowNisewongerBatraBrown,1988ChristiansenKrenk,2011JensenKupermanSchmidt}. With those thoughts, we consider the
following interface boundary value problem for the Helmholtz equation \cite{1984Brown, 2006KreissPetersson} which is posed in a bounded domain $\Omega\subseteq\mathbb{R}^2$: find $u(X)$ that satisfies the Helmholtz equation and the first-order absorbing boundary condition:
\commentout{
In this article, we intend to study a kind of second-order interface problem governed by Helmholtz equation. The governing equation could be derived from the wave equation with discontinuous coefficients~\cite{1984Brown, 2006KreissPetersson}
}
\begin{subequations}\label{model}
\begin{align}
\label{inter_PDE}
 -\nabla\cdot(\beta\nabla u)-k^2u=f, \;\;\;\; & \text{in} \; \Omega^-  \cup \Omega^+, \\
 \label{bnd_cond}
 \beta \frac{\partial{u}}{\partial{\boldsymbol{n}_{\Omega}}}+ik u=g, \;\;\;\; &\text{on} \; \partial\Omega,
\end{align}
together with the jump conditions across the interface
\cite{2007KlenowNisewongerBatraBrown,1984Brown, 1988ChristiansenKrenk, 2011JensenKupermanSchmidt,2006KreissPetersson}:
\begin{align}
[u]_{\Gamma} &:=u^-|_{\Gamma} - u^+|_{\Gamma}= 0, \label{jump_cond_1} \\
\big[\beta \nabla u\cdot \mathbf{n}\big]_{\Gamma} &:= \beta^- \nabla u^-\cdot \mathbf{n}|_{\Gamma} - \beta^+ \nabla u^+\cdot \mathbf{n}|_{\Gamma} = 0, \label{jump_cond_2}
\end{align}
where the domain $\Omega\subseteq\mathbb{R}^2$ is divided by an interface curve $\Gamma$ with $C^2$ smoothness into two subdomains $\Omega^-$ and $\Omega^+$,  with $\overline{\Omega}=\overline{\Omega^-\cup {\Omega^+}\cup\Gamma}$ ,  $u^s = u|_{\Omega^s}, s = \pm$ and $\mathbf{ n}$ is the unit normal vector to the interface $\Gamma$, $k$ is the wave number, $i=\sqrt{-1}$,  $\mathbf{n}_{\Omega}$ is the unit outward normal vector to $\partial \Omega$, and the coefficient $\beta$ is a piecewise positive constant function such that
\begin{equation*}
\beta(X)=
\left\{\begin{array}{cc}
\beta^- & \text{for} \; X\in \Omega^- ,\\
\beta^+ & \text{for} \; X\in \Omega^+.
\end{array}\right.
\end{equation*}
\end{subequations}
}

The Helmholtz boundary value problems without interface have been widely studied in the context of finite element methods, including classic finite element methods \cite{1995AzizWerschulz, 1997Babuska, 1995Babuska, 1994DouglasSheenSantos} and interior penalty Galerkin method~\cite{2016BurmanWuZhu, 2015DuWu, 2003Farhat, 2009FengWu, 2009GittelsonHiptmairPerugia, 2013MelenkParsaniaSauter, 2017LamShu, 2006Perugia}. These finite element methods can also be used to solve Helmholtz interface problems with
body-fitting meshes; however, their efficiency might be impeded for some applications where the locations and geometries of the interface keep changing because these conventional finite element methods require the meshes to be generated repeatedly according to the evolving interface configurations. To alleviate the burden of re-meshing, numerical methods based on interface-independent meshes have been developed, such as immersed interface methods (IIM) \cite{2006LiIto} and cut finite element methods (CutFEMs) \cite{2015BurmanClaus}. We refer the readers to ~\cite{2018Swift,2016ZhangLi,2017ZouWilkinsHarari} for applying those methods to Helmholtz interface problems.

The PPIFE methods developed in \cite{2019LinLinZhuang} are a class of finite element methods for solving
Helmholtz interface problems with interface-independent meshes. Extensive numerical experiments reported in \cite{2019LinLinZhuang} clearly demonstrate the optimal convergence of these PPIFE methods, and this observation motivates us to carry out related error analysis to theoretically confirm their optimal convergence.

The error analysis of traditional finite element methods for solving Helmholtz problems can be found in \cite{1979AzizKellogg, 1994DouglasSheenSantos, 1997Melenk}. Error analysis of traditional finite element methods for Helmholtz interface problems
can be found in recent publications \cite{2015Chaumont_thesis, 2016BarucpChaumontGout} where the stability and the optimal error bounds for the finite element solution were proved. We note that Schatz's argument~\cite{1974Schatz} is an important technique in the error analysis of
traditional finite element methods for Helmholtz problems, with or without interface.

\commentout{
For the error analysis of traditional finite element methods in solving Helmholtz problems, utilizing Schatz's argument~\cite{1974Schatz}, under the mesh constraint that $k^2h$ is sufficient small, the optimal error bounds with respect to $h$ were derived in some earlier literatures \cite{1979AzizKellogg, 1994DouglasSheenSantos, 1997Melenk}. In recent years, error analysis of traditional finite element methods for Helmholtz interface problems were also investigated with Schatz's argument, for example in \cite{2015Chaumont_thesis, 2016BarucpChaumontGout}, the authors analyzed the stability and proved the optimal error bounds for the finite element solution under the assumption that the mesh size and the error of the approximation of the wave speed are small enough based on a technique of approximation of the propagation media.
}

In this article, we also follow the framework of Schatz's argument to carry out the error analysis for the linear and bilinear symmetric
PPIFE methods developed in \cite{2019LinLinZhuang} for Helmholtz interface problems with a Robin boundary condition.
We utilize the coercivity and continuity of the bilinear form corresponding to the elliptic operator \cite{2018GuoLinZhuang,2015LinLinZhang}
to establish G{\aa}rding's inequality and continuity of the bilinear form in these PPIFE methods which are key ingredients in
Schatz's argument. We also derive a special trace inequality that is valid for IFE functions which are not $H^1$ functions in general. In particular, under suitable assumptions about the regularity of the exact solution and the mesh size, we are able to establish the optimal error bounds in both an energy norm and the standard $L^2$ norm for these PPIFE methods for solving the Helmholtz interface problems.

\commentout{
and the optimal orders of convergence were numerically observed when the mesh size is small enough, which motivates us to conduct the pertinent error analysis.

By constructing special local shape functions on the interface elements, IFE methods enable the interface to cut elements in a mesh, so they can solve the interface problems on the interface-independent meshes. We refer the readers to \cite{2016GuoLin,2009HeTHESIS, 2008HeLinLin, 2004LiLinLinRogers, 2003LiLinWu,2001LinLinRogersRyan, S.A.Sauter_R.Warnke} for some linear and bilinear  IFE methods based on
Cartesian meshes for elliptic interface problems. To alleviate the adverse impacts that might be brought by the discontinuity of IFE functions across the interface edges, partially penalized IFE (PPIFE) and discontinuous Galerkin IFE (DGIFE) methods~\cite{2017AdjeridGuoLin, 2009HeTHESIS,2014HeLinLin,2015LinLinZhang,2015LinYangZhang1,Zhang2013,2018ZhuangGuo} have been
developed for elliptic interface problems. As for PPIFE and DGIFE schemes for Helmholtz interface problems, our very recent work has addressed to them~\cite{2019LinLinZhuang}, where the optimal orders of convergence were numerically observed when the mesh size is small enough.

 For the error analysis of traditional finite elements in solving Helmholtz problems, by the framework of Schatz arguments~\cite{1974Schatz}, under the mesh constraint that $k^2h$ is sufficient small, the optimal error bounds with respect to $h$ were derived in some earlier literatures \cite{1979AzizKellogg, 1994DouglasSheenSantos, 1997Melenk}. In recent years, error analysis of traditional finite element methods for Helmholtz interface problems were also investigated with the Schatz argument, for example in \cite{2015Chaumont_thesis, 2016BarucpChaumontGout}, the authors analyzed the stability and proved the optimal error bounds for the finite element solution under the assumption that the mesh size and the error of the approximation of the wave speed are small enough based on a technique of approximation of the propagation media.

In this article, we aim to conduct error estimates of PPIFE methods for Helmholtz interface problems with Robin boundary condition. For this purpose, the main theoretical tool we develop in this article is a special trace inequality for the non-smooth functions (including IFE functions) on the whole domain. This tool enables us to estimate the errors of the PPIFE method based on the framework of Schatz argument. In particular, under suitable regularity assumption of the exact solution and the mesh constraint that $k^2h$ is sufficiently small, we establish the optimal error bounds in both the energy norm and $L^2$ norm for the Helmholtz interface problem.
}

The layout of the article is as follows: Section 2 introduces the notations and assumptions to be used in this article, and a symmetric PPIFE method
for the Helmholtz interface problem (\ref{model}) is described. In Section 3, optimal error bounds are derived for this PPIFE method in both the energy norm and $L^2$ norm. A numerical example is given in Section 4 to validate the theoretical results in Section 3.

\section{Notations and PPIFE Methods}

Without loss of generality, we assume that $\Omega\subset\mathbb{R}^2$ is a polygonal domain. We let $\mathcal{T}_h$ be a triangular or a rectangular mesh of $\Omega$ whose set of nodes and set of edges are $\mathcal{N}_h$ and $\mathcal{E}_h$, respectively. We assume the mesh $\mathcal{T}_h$ is independent of the interface $\Gamma$; hence, some elements of $\mathcal{T}_h$ will intersect with $\Gamma$ which will be called interface elements. We use $\mathcal{T}^i_h$ to denote the set of interface elements and use $\mathcal{T}^n_h$ for the set of non-interface elements. Similarly, we let $\mathcal{E}^i_h$ and $\mathcal{E}^n_h$ be the set of interface edges and the set of non-interface edges, respectively. In addition, we use
$\mathring{\mathcal{E}}_h$, $\mathring{\mathcal{E}}^i_h$ and  $\mathring{\mathcal{E}}_h^n$ for the set of interior edges, the set of interior interface edges, and the set of interior non-interface edges, respectively.

For each element $T\in\mathcal{T}_h$, we define its index set
as $\mathcal{I}_T=\{1,2,3\}$ when $T$ is triangular, but $\mathcal{I}_T=\{1,2,3,4\}$ when $T$ is rectangular.
Let $\psi_{j,T}$, $j\in \mathcal{I_T}$ be the standard linear or bilinear Lagrangian shape functions on $T\in\mathcal{T}_h$ such that
\begin{equation}
\label{stand_basis_fun}
\psi_{j,T}(A_i)=\delta_{ij}, ~~~ \forall i,j \in \mathcal{I}_T.
\end{equation}
We can use these shape functions to generate the following real polynomial spaces:
\begin{equation}
\label{usual_local_fem_sp}
\mathbb{\tilde{P}}(T)= \textrm{Span} \{ \psi_{j,T},j\in \mathcal{I}_T\} ~~~ \text{or} ~~~  \mathbb{\tilde{Q}}(T)= \textrm{Span} \{ \psi_{j,T},j\in \mathcal{I}_T\},~~\forall T\in\mathcal{T}_h,
\end{equation}
depending on whether $T$ is triangular or rectangular. The related complex polynomial spaces are:
\begin{equation}
\label{real_usual_local_fem_sp}
\mathbb{{P}}(T)= \{ v=v_1+iv_2:  v_1,v_2\in \mathbb{\tilde{P}}(T)\} ~~~ \text{or} ~~~  \mathbb{{Q}}(T)= \{ v=v_1+iv_2:  v_1,v_2\in \mathbb{\tilde{Q}}(T)\},~~\forall T\in\mathcal{T}_h.
\end{equation}
Throughout this article, without stating otherwise, the notation ``Span" represents spanning in the real number field.

Then, on every non-interface element $T \in \mathcal{T}_h^n$, the local real IFE space is the usual local real finite element space
\begin{equation}
\label{subspaces_real}
\tilde{S}_h(T)= \mathbb{\tilde{P}}(T)~~~ \text{or} ~~~ \tilde{S}_h(T)= \mathbb{\tilde{Q}}(T),~~\forall T\in\mathcal{T}_h^n,
\end{equation}
and its complex counterpart is
\begin{equation}
\label{subspaces}
{S}_h(T)= {{\mathbb{P}}}(T)~~~ \text{or} ~~~ {S}_h(T)= \mathbb{{Q}}(T),~~\forall T\in\mathcal{T}_h^n.
\end{equation}
In order to describe local IFE spaces on interface elements, we adopt the following standard assumptions on the mesh $\mathcal{T}_h$ of IFE spaces \cite{2009HeTHESIS, 2008HeLinLin, 2011HeLinLin}:
\begin{itemize}[leftmargin=30pt]
  \item [\textbf{(H1)}] The interface $\Gamma$ cannot intersect an edge of any element at more than two points unless the edge is part of $\Gamma$.
  \item [\textbf{(H2)}] If $\Gamma$ intersects the boundary of an element at two points, these intersection points must be on different edges of this element.
   \item [\textbf{(H3)}] The interface $\Gamma$ is a piecewise $C^2$ function, and for every interface element $T\in\mathcal{T}^i_h$, $\Gamma\cap T$ is $C^2$.
   \end{itemize}
\begin{figure}[H]
\centering
\begin{tabular}{c c}
\includegraphics[width=2in]{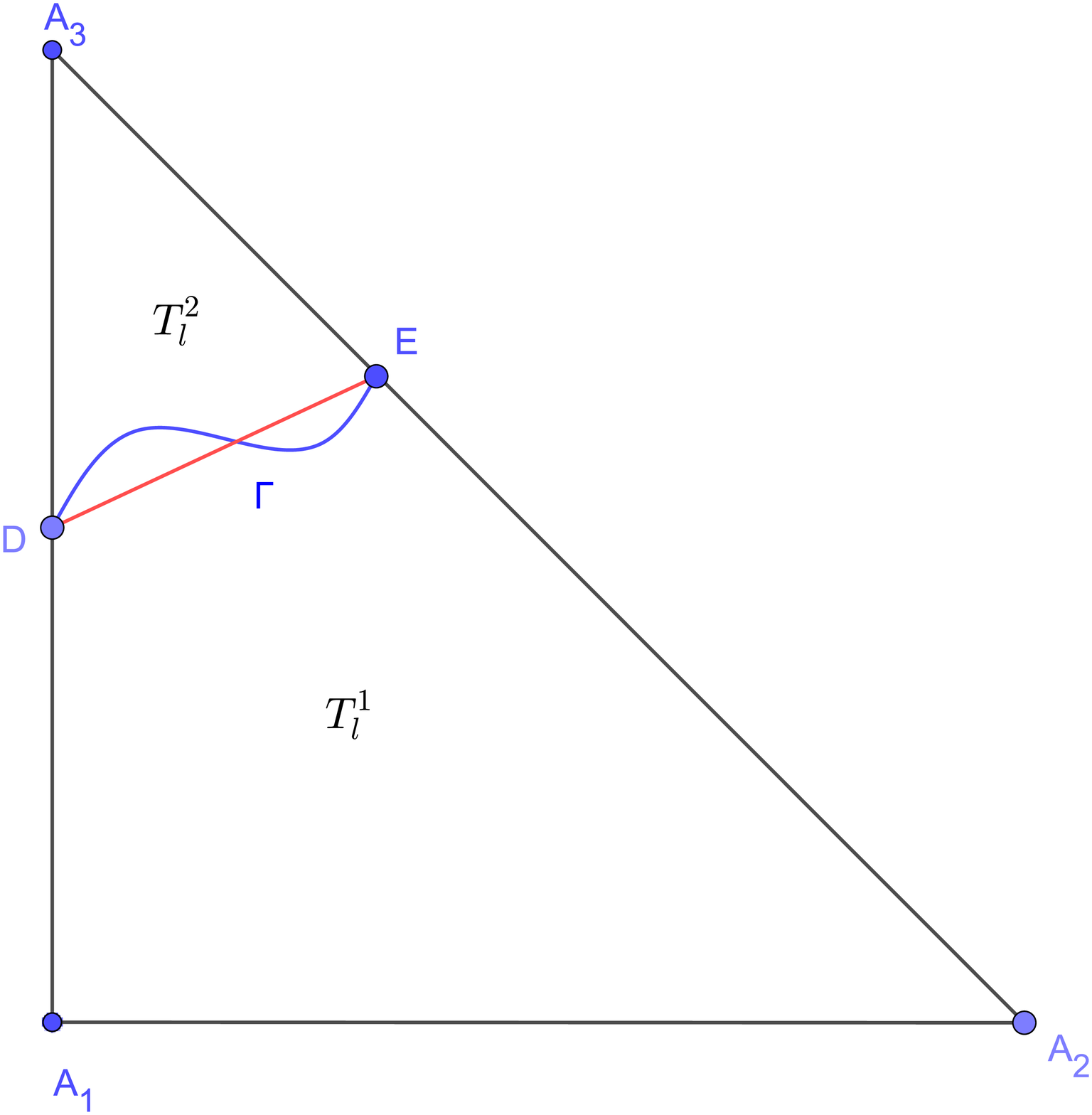}&
\includegraphics[width=2in]{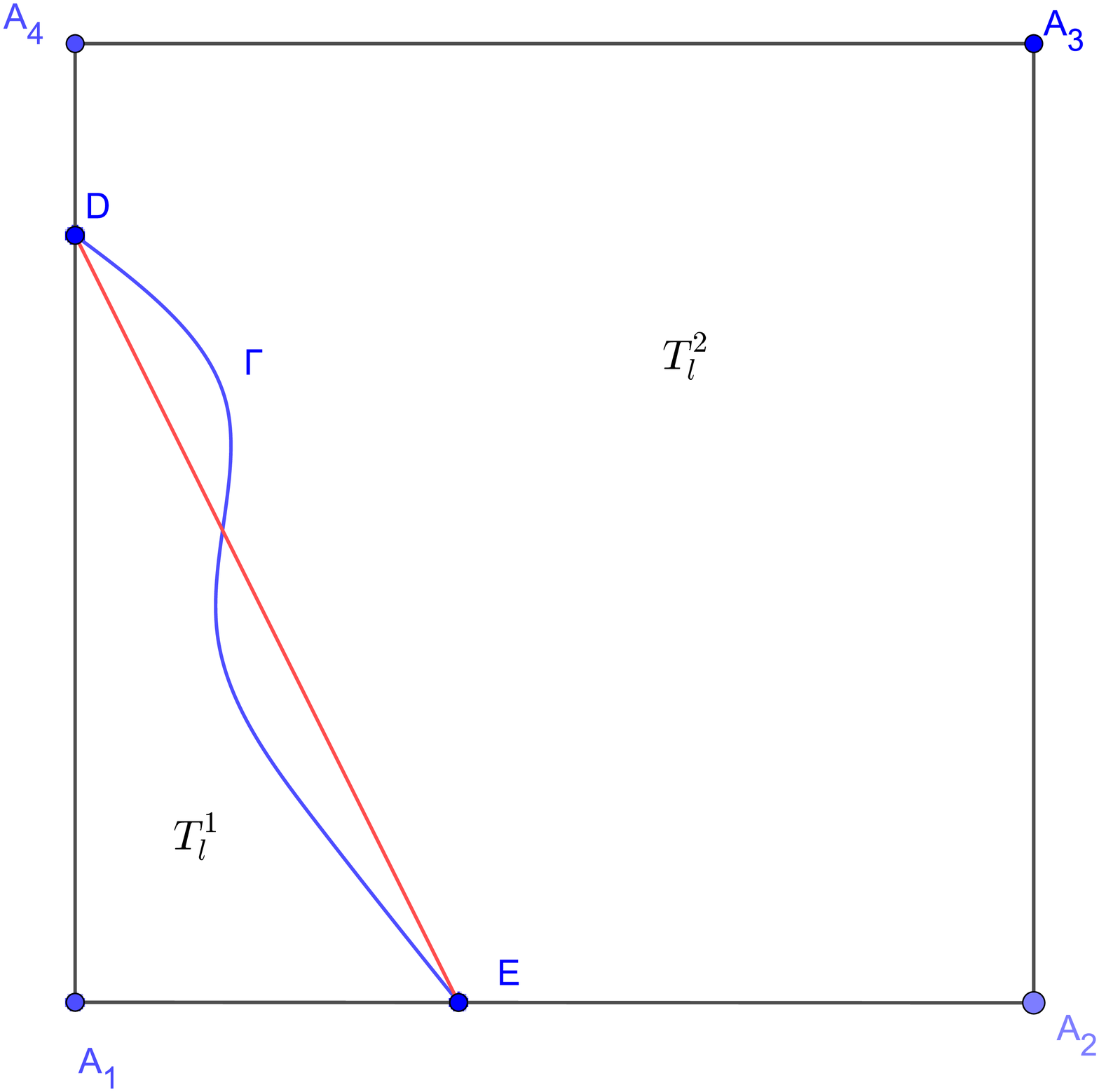}
\end{tabular}
\caption{left: triangular interface element; right: rectangular interface element.}
\label{inter_elem}
\end{figure}
Also, without loss of generality, we assume that $\Gamma \cap \partial \Omega = \emptyset$; hence, we can further assume that $\Gamma$ does not intersect with any boundary element because this assumption can be easily fulfilled when the mesh size is fine enough. To be specific, element
$T \in \mathcal{T}_h$ is a boundary element provided that $\abs{(\partial T) \cap (\partial \Omega)} \not = 0$, and we will use $\mathcal{T}_h^b$
to denote the collection of boundary elements of $\mathcal{T}_h$ from now on.

On the interface elements, we will use piecewise linear or bilinear polynomial as IFE shape functions \cite{2009HeTHESIS,2008HeLinLin,2004LiLinLinRogers}. For a typical interface element $T \in \mathcal{T}_h^i$ with vertices $A_i$, $i\in\mathcal{I}_T$, the interface partitions its index set $\mathcal{I_T}$ into $\mathcal{I}^-_T=\{ A_i: A_i\in T^- \}$ and $\mathcal{I}^+_T=\{ A_i: A_i\in T^+ \}$.
Furthermore, let $D$ and $E$ be the points where the interface $\Gamma$ intersects with the edges of $T$, as shown in Figure \ref{inter_elem}. Let $l$ be the line passing through $D$, $E$ with the normal vector $\bar{\mathbf{ n}}=(\bar{n}_x,\bar{n}_y)$. This line $l$ partitions $T$ into two subelements $T^{\pm}_l$.  A linear IFE function $\phi_{T} (x,y)$ on a triangular interface element $T$ is a piecewise linear polynomial in the following form \cite{2004LiLinLinRogers}:
\begin{equation}
\label{S_basis_fun_linear}
\phi_{T} (x,y) =
\begin{cases}
\phi_{T}^{-}(x,y)=a^-x+b^-y+c^-, ~~~ \textrm{if}~ (x,y)\in {T}^{-}_l,  \\
\phi_{T}^{+}(x,y)=a^+x+b^+y+c^+, ~~~ \textrm{if}~ (x,y)\in {T}^{+}_l,\\
\phi^-_T(D)=\phi^+_T (D), \phi^-_T (E)=\phi^+_T (E), \\
\beta^+ \frac{ \partial \phi^+_T }{\partial \bar{\mathbf{ n}} }-\beta^- \frac{ \partial \phi^-_T }{\partial \bar{\mathbf{ n}} }=0.
\end{cases}
\end{equation}
On a rectangular interface element $T$, a bilinear IFE function $\phi_{T} (x,y)$ is a piecewise bilinear polynomial in the following form \cite{2008HeLinLin}:
\begin{equation}
\label{S_basis_fun_bilinear}
\phi_{T} (x,y) =
\begin{cases}
\phi_{T}^{-}(x,y)=a^-x+b^-y+c^-+d^-xy, ~~~ \textrm{if}~ (x,y)\in T_l^{-},  \\
\phi_{T}^{+}(x,y)=a^+x+b^+y+c^++d^+xy, ~~~ \textrm{if}~ (x,y)\in T_l^{+},\\
\phi^-_T (D)=\phi^+_T (D), \phi^-_T (E)=\phi^+_T (E), d^-=d^+, \\
\int_{ \overline{DE}}(\beta^+ \frac{ \partial \phi^+_T }{\partial \bar{\mathbf{ n}} }-\beta^- \frac{ \partial \phi^-_T }{\partial \bar{\mathbf{ n}} })ds=0.
\end{cases}
\end{equation}
It has been proven \cite{2009HeTHESIS,2008HeLinLin,2004LiLinLinRogers} that IFE shape functions $\phi_{i,T}(x,y), i \in \mathcal{I}_T$ in the form of \eqref{S_basis_fun_linear} or \eqref{S_basis_fun_bilinear} can be uniquely constructed such that
\begin{equation}
\label{ife_basis_fun}
\phi_{i,T}(A_j)=\delta_{ij}, ~~~ \forall i,j \in \mathcal{I}_T.
\end{equation}
Then, we define the local real IFE space on the interface element $T \in \mathcal{T}_h^i$ as
\begin{equation}{\label{ife_shape_fun_real}}
\tilde{S}_h(T)=\mathrm {Span}\{\phi_{i,T}, i\in \mathcal{ I}_T \},~~\forall T\in \mathcal{T}_h^i,
\end{equation}
and its complex counterpart is
\begin{equation}{\label{ife_shape_fun}}
{S}_h(T)= \{ v=v_1+iv_2:  v_1,v_2\in {\tilde{S}_h}(T)\},~~\forall T\in \mathcal{T}_h^i.
\end{equation}
Using the real local IFE spaces in \eqref{subspaces_real} and \eqref{ife_shape_fun_real}, we define the global real IFE space as follows:
\begin{equation}
\begin{split}
\label{global_space_real}
\tilde{S}_{h}(\Omega)=&\left\{v\in L^2(\Omega)~:~ v|_{T}\in \tilde{S}_{h}(T),~~\forall T \in \mathcal{T}_h,  v\,\, \text{is continuous at each} \, A\in\mathcal{N}_h \right\},
\end{split}
\end{equation}
and its complex counterpart is defined with the local complex IFE spaces in \eqref{subspaces} and \eqref{ife_shape_fun}:
\begin{equation}
\begin{split}
\label{disconti_global_space}
{S}_{h}(\Omega)=&\left\{v\in L^2(\Omega)~:~ v|_{T}\in {S}_{h}(T),~~\forall T \in \mathcal{T}_h,  v\,\, \text{is continuous at each} \, A\in\mathcal{N}_h \right\}.
\end{split}
\end{equation}
In the discussion below, we will use the following function spaces. For $\tilde \Omega \subseteq \Omega$, let $H^p(\tilde \Omega), p \geq 0$ be the standard Sobolev space on $\tilde \Omega$ with the norm $\|\cdot\|_{p,\tilde{\Omega}}$ and semi-norm $|\cdot|_{p,\tilde{\Omega}}$. Furthermore, if $\tilde \Omega^s = \tilde \Omega \cap \Omega^s \not = \emptyset$, $s = \pm$, we let
\begin{equation*}
PH^2(\tilde \Omega)=
\{u: u|_{\tilde{\Omega}^s}\in H^2(\tilde \Omega^s),\; s=\pm; \; [u]=0, \; [\beta\nabla u\cdot \mathbf{n}_\Gamma]=0\; \text{on}\; \Gamma \cap \tilde \Omega  \},
\end{equation*}
equipped with the following norms and semi-norms:
\begin{equation*}
\|\cdot\|_{2,\tilde \Omega}=\sqrt{\|\cdot\|^2_{2,\tilde \Omega^-}+\|\cdot\|^2_{2,\tilde \Omega^+}}, \;\;\;\;\; |\cdot|_{2,\tilde \Omega}=\sqrt{|\cdot|^2_{2,\tilde \Omega^-}+|\cdot|^2_{2,\tilde \Omega^+}}.
\end{equation*}
\commentout{
In the description and analysis of the PPIFE methods, we will use the following space defined according to the mesh $\mathcal{T}_h$:
\begin{equation}
\begin{split}
\label{conti_sp}
V_h(\Omega) = \{ v\in L^2(\Omega):  v|_T\in H^1(T), \,\, \nabla v \cdot \boldsymbol{n}|_{\partial T}\in L^2(\partial T), \,\, v\,\, \text{is continuous across each}\, e\in \mathring{\mathcal{E}}_h^n \},
\end{split}
\end{equation}
and it can be easily verify that $S_h(\Omega)\subset V_h(\Omega)$.
}

\commentout{
For each interface element $T\in \mathcal{T}^i_h$, the interface $\Gamma$ straddles it into two subelements $T^1=\Omega^1\cap T$ and $T^2=\Omega^2\cap T$, splitting the index set $\mathcal{I}$ into two subsets: $\mathcal{I}^1=\{ i\,:\,N_i\in T^1\} $ and $\mathcal{I}^2=\{ i\,:\,N_i\in T^2\}$, with $\mathcal{I}= \mathcal{I}^1 \cup \mathcal{I}^2$. Therefore, the IFE shape functions are chosen from the following space
\begin{equation}
\label{piecewise_poly}
\mathcal{P}^p(T) = \{ q\, :\, q|_{T^1}\in\mathbb{P}^p(T^1)~~\textrm{and}~~q|_{T^2}\in\mathbb{P}^p(T^2) \},\quad p\geq 1
\end{equation}
on the triangular interface element $T$, and
\begin{equation}
\label{piecewise_poly}
\mathcal{Q}^1(T) = \{ q\, :\, q|_{T^1}\in\mathbb{Q}^1(T^1)~~\textrm{and}~~q|_{T^2}\in\mathbb{Q}^1(T^2) \},
\end{equation}
on the rectangular interface element $T$.

In $\mathcal{T}_h$, we denote the sets of interface elements and non-interface elements by $\mathcal{T}^i_h$ and ${T}^n_h$. For each element $T\in\mathcal{T}_h$, if $\mathcal{T}_h$ is a triangular mesh, in order to construct $p$-th degree local IFE space,   we introduce the index set $\mathcal{I}=\{1,2,...,\frac{(p+1)(p+2)}{2}\}$; if $\mathcal{T}_h$ is a rectangular mesh, then the index set $\mathcal{I}=\{1,2,3,4\}$. And denote $N_i\in \mathcal{N}_h$, $i \in \mathcal{I}$ as the usual local nodes on $T$, with $\mathcal{N}_h$ as the set of local nodes in all elements. Let $\psi_{j,T}$, $j\in \mathcal{I}$ be the standard $p$-th degree or bilinear finite element nodal basis function on $T$ , i.e.,
\begin{equation}
\label{stand_basis_fun}
\psi_{j,T}(N_i)=\delta_{ij}, ~~~ \forall i,j \in \mathcal{I}.
\end{equation}
Then we use the local $p$-th ($p\geq 1$) degree finite element space
\begin{equation}
\label{subspaces}
{\mathbb{P}}^p(T)= \textrm{Span} \{ \psi_{j,T},j\in \mathcal{I}. \}
\end{equation}
as local IFE space on the non-interface triangular element $T$, here $\psi_{j,T}$ is the $p$-th degree finite element nodal basis function. Also we employ the local bilinear finite element space
\begin{equation}
\label{subspaces}
{\mathbb{Q}}^1(T)= \textrm{Span} \{ \psi_{j,T},j\in \mathcal{I} \}
\end{equation}

}

Now, we present the PPIFE methods developed in \cite{2019LinLinZhuang} for the Helmholtz interface problem in order to analyze them.
The description of these PPIFE methods relies on the following space defined according to the mesh $\mathcal{T}_h$:
\begin{equation}
\begin{split}
\label{conti_sp}
V_h(\Omega) = \{ v\in L^2(\Omega):  v|_T\in H^1(T), \,\, \nabla v \cdot \boldsymbol{n}|_{\partial T}\in L^2(\partial T) ~~\forall T \in \mathcal{T}_h, \,\, [v]_e = 0
~~\forall e\in \mathring{\mathcal{E}}_h^n \},
\end{split}
\end{equation}
and it can be easily verified that $S_h(\Omega)\subset V_h(\Omega)$.
We will also employ the following standard notations for penalty terms on edges of the mesh $\mathcal{T}_h$: on each $e\in\mathring{\mathcal{E}}_h$ shared by two elements $T^e_1$ and $T^e_2$, let
\begin{equation}
\label{ed_ope}
[v]_e = v|_{T^e_1} - v|_{T^e_2}, ~~~ \text{and} ~~~ \{ v \}_e = \frac{1}{2} \left( v|_{T^e_1} + v|_{T^e_2} \right),~~\forall v\in V_h(\Omega).
\end{equation}
According to \cite{2019LinLinZhuang}, the Helmholtz interface problem (\ref{model}) has a weak form as follows:
\begin{equation}\label{weak_form_pp}
b_h(u, v) = L_f(v)+ (g,v)_{\partial \Omega},  ~~~~ \forall v \in V_h(\Omega),  \\
\end{equation}
in which the bilinear form $b_h(\cdot, \cdot)$ is such that
\begin{equation}
b_h(u, v)=a_h(u,v)+ik(u,v)_{\partial \Omega}-k^2(u, v)_\Omega, ~~\forall u, v \in V_h(\Omega), \label{eq:bilinear_form_b_h}
\end{equation}
where $(\cdot, \cdot)_{\partial \Omega}$ and $(\cdot, \cdot)_\Omega$ are the standard $L^2$ inner products of the two involved functions on $\partial \Omega$ and $\Omega$, respectively, $a_h(\cdot, \cdot): V_h(\Omega) \times V_h(\Omega) \rightarrow \mathbb{C}$
is a bilinear form defined as
\begin{equation} \label{eq:bilinear_form_pp}
\begin{split}
a_h(u,v)=& \sum_{T\in\mathcal{T}_h} \int_T \beta \nabla u\cdot \nabla \overline{v} dX
 -  \sum_{e\in\mathring{\mathcal{E}}_h^i} \int_e \{ \beta \nabla u\cdot \mathbf{ n}_e \}_e [\overline{v}]_e ds  \\
 &-\sum_{e\in\mathring{\mathcal{E}}_h^i} \int_e \{ \beta \nabla \overline{v}\cdot \mathbf{ n}_e \}_e [u]_e ds+i\sum_{e\in\mathring{\mathcal{E}}_h^i} \frac{\sigma^0_e}{|e|} \int_e [u]_e\,[\overline{v}]_e ds, \hspace{0.3in} \forall u,v\in V_h(\Omega),
\end{split}
\end{equation}
and  $L_f(\cdot): V_h (\Omega) \rightarrow \mathbb{C}$ is a linear form defined as
\begin{eqnarray} \label{eq:ppife_linear_functional}
L_f(v) = \int_{\Omega} f\overline{v}  dX=(f,v)_{\Omega}, \hspace{0.3in} \forall v\in V_h(\Omega).
\end{eqnarray}
As usual, the weak form \eqref{weak_form_pp} and the fact that $S_h(\Omega)\subset V_h(\Omega)$ lead to the symmetric PPIFE methods for the Helmholtz interface problem (\ref{model}):
find $u_h\in S_{h}(\Omega)$,  such that
\begin{equation}{\label{weak_form_ppife}}
b_h(u_h, v_h)= L_f(v_h)+ (g,v_h)_{\partial \Omega},  ~~~~ \forall v_h \in S_{h}(\Omega).
\end{equation}

\commentout{
\subsection{Linear/bilinear IFE methods}
\subsubsection{Classic linear/bilinear IFE method}
To obtain a classic IFE formulation of the interface problem, we multiply Eqn.(\ref{inter_PDE}) by a test function $v\in V_h(\Omega)$ and integrate on the domain $\Omega$, using Green's formula, it follows
\begin{equation}{\label{first_int}}
\int_{\Omega} \beta \nabla u \nabla \bar{v} dX-\int_{\partial{\Omega^1}} \beta \nabla u \cdot \boldsymbol{n_{\Omega^1}} \bar{v} ds-\int_{\partial{\Omega^2}} \beta \nabla u \cdot \boldsymbol{n_{\Omega^2}} \bar{v} ds-w^2\int_{\Omega} u \bar{v} dX=\int_{\Omega} f\bar{v}  dX,
\end{equation}
it follows
\begin{equation}{\label{first_int}}
\int_{\Omega} \beta \nabla u \nabla \bar{v} dX-\int_{\partial_{\Omega}} \beta \nabla u \cdot \boldsymbol{n_{\Omega}} \bar{v} ds-w^2\int_{\Omega} u \bar{v} dX=\int_{\Omega} f\bar{v}  dX,
\end{equation}
rearrange the terms and apply the boundary condition~(\ref{bnd_cond})
\begin{equation}{\label{classic_rough_wform}}
\sum_{T\in\mathcal{T}_h}\int_T \beta \nabla u \nabla \bar{v} dX+iw\sum_{e\in \mathcal{E}_h^b}\int_{e}u\bar{v}ds-w^2\int_{\Omega} u \bar{v} dX =\int_{\Omega} f\bar{v}  dX+\sum_{e\in \mathcal{E}_h^b}\int_{e}g\bar{v}  ds,
\end{equation}
then the weak form of the interface problem in (\ref{model}) to (\ref{jump_cond_2}): find $u\in PH^2(\Omega)$ such that
\begin{equation}
\label{weak_form}
a_h(u,v)-w^2(u, v)_\Omega = L_f(v),  ~~~~ \forall v \in V_h(\Omega),  \\
\end{equation}
where the bilinear form  $a_h:V_h(\Omega)\times V_h(\Omega)\rightarrow \mathbb{R}$ is defined such that for all $u,v\in V_h(\Omega)$,

\begin{equation}
\label{bilinear_form}
a_h(u,v)=\sum_{T\in\mathcal{T}_h} \int_T \beta \nabla u\cdot \nabla \bar{v} dX +iw  \sum_{e\in {\mathcal{E}_h^b}}\int_e u \bar{v} ds.
\end{equation}
and a linear form $L_f: V_h(\Omega)\rightarrow \mathbb{R}$ such that $v\in V_h(\Omega)$,
\begin{equation}
\label{linearform}
L_f(v) = \int_{\Omega} f \bar{v} dX + \sum_{e\in\mathcal{E}^b_h} \int_e \bar{v}g ds.
\end{equation}

Then the linear or bilinear classic IFE method for the given Helmholtz type interface problem in (\ref{model})-(\ref{jump_cond_2})  can be formulated as:  find $u_h\in S^1_{h}(\Omega)$ such that
\begin{subequations}
\label{dgife}
\begin{align}
      & a_h(u_h,v_h)-w^2(u_h, v_h)_\Omega = L_f(v_h),  ~~~~ \forall v_h \in S^1_{h}(\Omega),  \\
      & \beta\frac{\partial{u_h}}{{\partial{\bf{n}_{\Omega}}}}(X)+iw u_h (X) = g(X), ~~~~~ ~~~~~ ~\forall X \in \partial \Omega \cap \mathcal{N}_h.
\end{align}
\end{subequations}
\subsubsection{Linear/bilinear PPIFE method}
For the PPIFE formulation, multiply Eq.(\ref{inter_PDE}) by a test function $v\in V_h(\Omega)$ and integrate on each element $T\in \mathcal{T}_h$, with Green's formula, it follows:
\begin{equation}{\label{first_int}}
\int_T \beta \nabla u \nabla \bar{v} dX-\int_{\partial_T} \beta \nabla u \cdot \boldsymbol{n_T} \bar{v} ds-w^2\int_T u \bar{v} dX=\int_T f\bar{v}  dX,
\end{equation}
sum Eq.(\ref{first_int}) over all $T\in \mathcal{T}_h$, it follows:
\begin{equation}{\label{sum_over_ele}}
\sum_{T\in\mathcal{T}_h}\int_T \beta \nabla u \nabla \bar{v} dX-\sum_{e\in \mathring{\mathcal{E}_h}}\int_{e} \{\beta \nabla u \cdot \boldsymbol{n_e} \}[\bar{v}] ds-\sum_{e\in \mathcal{E}_h^b}\int_{e} \beta \nabla u \cdot \boldsymbol{n_e} \bar{v} ds-w^2\int_{\Omega} u \bar{v} dX=\int_T f\bar{v}  dX,
\end{equation}
note that the fact
\begin{equation}
\{ \beta \nabla u \cdot\boldsymbol{n_e}\}_e=(\beta \nabla u\cdot \boldsymbol{n_e})|_e,~~~~\forall e\in \mathring{\mathcal{E}}_h^n,
\end{equation}
it follows
\begin{equation}
\sum_{T\in\mathcal{T}_h}\int_T \beta \nabla u \nabla \bar{v} dX-\sum_{e\in \mathring{\mathcal{E}_h^i}}\int_{e} \{\beta \nabla u \cdot \boldsymbol{n_e} \}[\bar{v}] ds-\sum_{e\in \mathcal{E}_h^b}\int_{e} \beta \nabla u \cdot \boldsymbol{n_e} \bar{v} ds-w^2\int_{\Omega} u \bar{v} dX=\int_T f\bar{v}  dX,
\end{equation}
rearrange the terms and utilize the boundary condition (\ref{bnd_cond})
\begin{equation}{\label{rough_pp_wform}}
\sum_{T\in\mathcal{T}_h}\int_T \beta \nabla u \nabla \bar{v} dX-\sum_{e\in \mathring{\mathcal{E}_h^i}}\int_{e} \{\beta \nabla u \cdot \boldsymbol{n_e} \}[\bar{v}] ds+iw\sum_{e\in \mathcal{E}_h^b}\int_{e}u\bar{v}ds-w^2\int_{\Omega} u \bar{v} dX =\int_{\Omega} f\bar{v}  dX+\sum_{e\in \mathcal{E}_h^b}\int_{e}g\bar{v}  ds,
\end{equation}
note the regularity of $u$, for any parameter $\epsilon$, and $\sigma_e^0\geq 0$, we have\cite{2015LinLinZhang}
\begin{equation}{\label{reg_pp_cancel}}
\epsilon \sum_{e\in\mathring{\mathcal{E}}_h^i} \int_e \{ \beta \nabla \bar{v}\cdot \mathbf{ n}_e \}_e [u]_e ds=0,\quad i\sum_{e\in\mathring{\mathcal{E}}_h^i} \frac{\sigma^0_e}{|e|} \int_e [u]_e\,[\bar{v}]_e ds=0,
\end{equation}
where $i=\sqrt{-1}$, add the two terms in (\ref{reg_pp_cancel}) to (\ref{rough_pp_wform}), it forms the weak form of the considered interface problem:
find $u\in PH^2(\Omega)$ such that
\begin{equation}
\label{weak_form_pp}
a_h^{PP}(u,v)-w^2(u, v)_\Omega = L_f(v),  ~~~~ \forall v \in V_h(\Omega),  \\
\end{equation}
where the bilinear form  $a_h^{PP}:V_h(\Omega)\times V_h(\Omega)\rightarrow \mathbb{R}$ is defined such that for all $u,v\in V_h(\Omega)$,
\begin{equation}
\begin{split}
a_h^{PP}(u,v)=& \sum_{T\in\mathcal{T}_h} \int_T \beta \nabla u\cdot \nabla \bar{ v} dX \\
 - & \sum_{e\in\mathring{\mathcal{E}}_h^i} \int_e \{ \beta \nabla u\cdot \mathbf{ n}_e \}_e [\bar{v}]_e ds  + \epsilon \sum_{e\in\mathring{\mathcal{E}}_h^i} \int_e \{ \beta \nabla \bar{v}\cdot \mathbf{ n}_e \}_e [u]_e ds\\
 +&iw  \sum_{e\in {\mathcal{E}_h^b}}\int_e u \bar{v} ds+i\sum_{e\in\mathring{\mathcal{E}}_h^i} \frac{\sigma^0_e}{|e|} \int_e [u]_e\,[\bar{v}]_e ds,
\end{split}
\end{equation}
and $L_f$ is the same as it is defined in (\ref{linearform}).
Then the linear or bilinear PPIFE method for the Helmholtz interface problem considered is formulated as:  find $u_h\in S^1_{h}(\Omega)$ such that
\begin{subequations}
\label{dgife}
\begin{align}
      & a_h^{PP}(u_h,v_h)-w^2(u_h, v_h)_\Omega = L_f(v_h),  ~~~~ \forall v_h \in S^1_{h}(\Omega),  \\
      & \beta\frac{\partial{u_h}}{{\partial{\bf{n}_{\Omega}}}}(X)+iw u_h (X) = g(X), ~~~~~ ~~~~~ ~~~ \forall X \in \partial \Omega \cap \mathcal{N}_h.
\end{align}
\end{subequations}
\subsubsection{Linear/bilinear DGIFE method}
The derivation of DG formulation is very similar to that of partial penalized formulation, the weak form of the interface problem (\ref{model})-(\ref{jump_cond_2}) can be described as: find $u\in PH^2(\Omega)$ such that
\begin{equation}
\label{weak_form}
a_h^{DG}(u,v)-w^2(u, v)_\Omega = L_f(v),  ~~~~ \forall v \in DV_h(\Omega),  \\
\end{equation}
where the bilinear form  $a_h^ {DG}:DV_h(\Omega)\times DV_h(\Omega)\rightarrow \mathbb{R}$ is defined such that for all $u,v\in DV_h(\Omega)$,
\begin{equation}
\begin{split}
a_h^{DG}(u,v)=& \sum_{T\in\mathcal{T}_h} \int_T \beta \nabla u\cdot \nabla \bar{v} dX \\
 - & \sum_{e\in\mathring{\mathcal{E}}_h} \int_e \{ \beta \nabla u\cdot \mathbf{ n}_e \}_e [\bar{v}]_e ds  + \epsilon \sum_{e\in\mathring{\mathcal{E}}_h} \int_e \{ \beta \nabla \bar{v}\cdot \mathbf{ n}_e \}_e [u]_e ds\\
 +&iw  \sum_{e\in {\mathcal{E}_h^b}}\int_e u \bar{v} ds+i\sum_{e\in\mathring{\mathcal{E}}_h} \frac{\sigma^0_e}{|e|} \int_e [u]_e\,[\bar{v}]_e ds.
\end{split}
\end{equation}
and the linear form $L_f: DV_h(\Omega)\rightarrow \mathbb{R}$ such that for $v\in DV_h(\Omega)$, there holds (\ref{linearform}).
Thus the linear or bilinear DGIFE formulation is:  find $u_h\in DS^1_{h}(\Omega)$ such that
\begin{subequations}
\label{dgife}
\begin{align}
      & a_h^{DG}(u_h,v_h)-w^2(u_h, v_h)_\Omega = L_f(v_h),  ~~~~ \forall v_h \in DS^1_{h}(\Omega),  \\
      & \beta\frac{\partial{u_h}}{{\partial{\bf{n}_{\Omega}}}}(X)+iw u_h (X) = g(X), ~~~~~ ~~~~~ ~~ \forall X \in \partial \Omega \cap \mathcal{N}_h.
\end{align}
\end{subequations}
\subsection{Higher degree IFE methods}
\subsubsection{Higher degree PPIFE methods}
To derive a PPIFE scheme to be used in the higher order IFE spaces, multiply Eq.(\ref{inter_PDE}) by a test function $v\in W_h(\Omega)$ and integrate on each element $T\in \mathcal{T}_h$, with Green's formula, it follows:
\begin{equation}{\label{first_int_hd}}
\int_T \beta \nabla u \nabla \bar{v} dX-\int_{\partial_T} \beta \nabla u \cdot \boldsymbol{n_T} \bar{v} ds-\int_{\Gamma_T} \{\beta \nabla u \cdot \boldsymbol{n_\Gamma}\} [\bar{v}] ds-w^2\int_T u \bar{v} dX=\int_T f\bar{v}  dX,
\end{equation}
sum Eq.(\ref{first_int_hd}) over all $T\in \mathcal{T}_h$, it follows:
\begin{equation}{\label{sum_over_ele}}
\sum_{T\in\mathcal{T}_h}\int_T \beta \nabla u \nabla \bar{v} dX-\sum_{e\in \mathring{\mathcal{E}_h}}\int_{e} \{\beta \nabla u \cdot \boldsymbol{n_e} \}[\bar{v}] ds-\sum_{e\in \mathcal{E}_h^b}\int_{e} \beta \nabla u \cdot \boldsymbol{n_e} \bar{v} ds-\int_{\Gamma}\{ \beta \nabla u\cdot \boldsymbol{n_\Gamma}\} [\bar{v}] ds-w^2\int_{\Omega} u \bar{v} dX=\int_T f\bar{v}  dX,
\end{equation}
note that the fact
\begin{equation}
\{ \beta \nabla u \cdot\boldsymbol{n_e}\}_e=(\beta \nabla u\cdot \boldsymbol{n_e})|_e,~~~~\forall e\in \mathring{\mathcal{E}}_h^n,
\end{equation}
it follows
\begin{equation}
\sum_{T\in\mathcal{T}_h}\int_T \beta \nabla u \nabla \bar{v} dX-\sum_{e\in \mathring{\mathcal{E}_h^i}}\int_{e} \{\beta \nabla u \cdot \boldsymbol{n_e} \}[\bar{v}] ds-\sum_{e\in \mathcal{E}_h^b}\int_{e} \beta \nabla u \cdot \boldsymbol{n_e} \bar{v} ds-\int_{\Gamma} \{\beta \nabla u \cdot \boldsymbol{n_{\Gamma}}\} [\bar{v}] ds-w^2\int_{\Omega} u \bar{v} dX=\int_T f\bar{v}  dX,
\end{equation}
rearrange the terms and utilize the boundary condition (\ref{bnd_cond})
\begin{equation}{\label{rough_pp_wform_hd}}
\begin{split}
\sum_{T\in\mathcal{T}_h}\int_T \beta \nabla u \nabla \bar{v} dX-\sum_{e\in \mathring{\mathcal{E}_h^i}}\int_{e} \{\beta \nabla u \cdot \boldsymbol{n_e} \}[\bar{v}] ds+iw\sum_{e\in \mathcal{E}_h^b}\int_{e}u\bar{v}ds-\int_{\Gamma} \{\beta \nabla u \cdot \boldsymbol{n_\Gamma}\} [\bar{v}] ds-w^2\int_{\Omega} u \bar{v} dX\\
=\int_{\Omega} f\bar{v}  dX+\sum_{e\in \mathcal{E}_h^b}\int_{e}g\bar{v}  ds,
\end{split}
\end{equation}
note the regularity of $u$, for any parameter $\epsilon$, and $\sigma_e^0\geq 0$, we have\cite{2015LinLinZhang}
\begin{equation}{\label{reg_pp_cancel_hd}}
\begin{split}
\epsilon \sum_{e\in\mathring{\mathcal{E}}_h^i} \int_e \{ \beta \nabla \bar{v}\cdot \mathbf{ n}_e \}_e [u]_e ds=0,\quad i\sum_{e\in\mathring{\mathcal{E}}_h^i} \frac{\sigma^0_e}{|e|} \int_e [u]_e\,[\bar{v}]_e ds=0, \\
 \epsilon \int_{\Gamma} \{ \beta \nabla \bar{v}\cdot \mathbf{ n}_\Gamma \}_\Gamma [u]_\Gamma ds=0, \quad i\frac{\sigma^0_e}{|e|} \int_\Gamma [u]_\Gamma\,[\bar{v}]_\Gamma ds=0,
\end{split}
\end{equation}
add the four terms in (\ref{reg_pp_cancel_hd}) to (\ref{rough_pp_wform_hd}), it forms the weak form of the considered interface problem (\ref{model})-(\ref{ext_jump_cond_1}):
find $u\in PH^2(\Omega)$ such that
\begin{equation}
\label{weak_form_pp}
a_h^{PP}(u,v)-w^2(u, v)_\Omega = L_f(v),  ~~~~ \forall v \in W_h(\Omega),  \\
\end{equation}
where the bilinear form  $a_h^{PP}:W_h(\Omega)\times W_h(\Omega)\rightarrow \mathbb{R}$ is defined such that for all $u,v\in W_h(\Omega)$,
\begin{equation}
\begin{split}
a_h^{PP}(u,v)=& \sum_{T\in\mathcal{T}_h} \int_T \beta \nabla u\cdot \nabla \bar{ v} dX \\
 - & \sum_{e\in\mathring{\mathcal{E}}_h^i} \int_e \{ \beta \nabla u\cdot \mathbf{ n}_e \}_e [\bar{v}]_e ds  + \epsilon \sum_{e\in\mathring{\mathcal{E}}_h^i} \int_e \{ \beta \nabla \bar{v}\cdot \mathbf{ n}_e \}_e [u]_e ds\\
 +&iw  \sum_{e\in {\mathcal{E}_h^b}}\int_e u \bar{v} ds+i\sum_{e\in\mathring{\mathcal{E}}_h^i} \frac{\sigma^0_e}{|e|} \int_e [u]_e\,[\bar{v}]_e ds-\int_{\Gamma} \{\beta \nabla u \cdot \boldsymbol{n_\Gamma}\} [\bar{v}] ds+ \\
& \epsilon \int_{\Gamma} \{ \beta \nabla \bar{v}\cdot \mathbf{ n}_\Gamma \}_\Gamma [u]_\Gamma ds+i\frac{\sigma^0_e}{|e|} \int_\Gamma [u]_\Gamma\,[\bar{v}]_\Gamma ds,
\end{split}
\end{equation}
and the linear form $L_f: W_h(\Omega)\rightarrow \mathbb{R}$ is defined to such that for $v\in W_h(\Omega)$, there holds (\ref{linearform}).
Then the higher degree PPIFE method for the Helmholtz interface problem (\ref{model}) to (\ref{ext_jump_cond_1}) can be formulated as: find $u_h\in S_{h}^p(\Omega)$, $p\geq 2$ such that
\begin{align}
      & a_h^{PP}(u_h,v_h)-w^2(u_h, v_h)_\Omega = L_f(v_h),  ~~~~ \forall v_h \in S^p_{h}(\Omega),  \\
      & \beta\frac{\partial{u_h}}{{\partial{\bf{n}_{\Omega}}}}(X)+iw u_h (X) = g(X), ~~~~~ ~~~~~ ~~~ \forall X \in \partial \Omega \cap \mathcal{N}_h.
\end{align}
}


\section{Error Analysis of Symmetric PPIFE Methods}

The error estimation to be presented for the symmetric PPIFE methods  described by \eqref{weak_form_ppife} will use the following three energy norms and the broken $H^1$ norm for functions $v\in V_h(\Omega)$:
\begin{equation}
\Vert v\Vert_h^2=\sum_{T\in\mathcal{T}_h} \int_T \beta \Vert \nabla v \Vert^2 dX+\sum_{e\in \mathring{\mathcal{E}_h^i}}{\sigma_e^0}\int_e \Vert\vert e \vert^{-1/2} [v]\Vert^2 ds,
\end{equation}

\begin{equation}
\vertiii{v}_h^2=\Vert v \Vert_h^2+\sum_{e\in \mathring{\mathcal{E}_h^i}}({\sigma_e^0})^{-1}\int_e\Vert \vert e \vert^{1/2}\{\beta \nabla v\cdot \boldsymbol{n_e}\}\Vert^2ds,
\end{equation}

\begin{equation}
\vertiii{v}_{\mathcal{H}}^2=\vertiii{v}_h^2+ k^2 \Vert v \Vert_{L^2(\Omega)}^2,
\end{equation}

\begin{equation}
\Vert v\Vert_{1,\Omega}^2=\sum_{T\in\mathcal{T}_h} \Vert v \Vert_{1,T}^2,~~\vert v\vert_{1,\Omega}^2=\sum_{T\in\mathcal{T}_h} \vert v \vert_{1,T}^2.
\end{equation}
First we make the following assumption on the regularity of the exact solution
\begin{asp} {\label{regularity_assumption}}
 Assume that the exact solution $u$ to the interface problem (\ref{model}) is in $PH^2(\Omega)$ and the following estimate holds for some constant $C$:
 \begin{equation}\label{aspeq}
\Vert u \Vert_{2,\Omega}\leq C (k+k^{-1}) (\Vert f \Vert_{L^2(\Omega)}+\Vert g \Vert_{L^2(\partial \Omega)}).
\end{equation}
\end{asp}
Assumption \ref{regularity_assumption} can be satisfied when the $\partial \Omega$ and $\Gamma$ are sufficiently smooth,
see \cite{2009FengWu, 1997Melenk} and \cite{2017MoiolaSpence} for more details.
Next we recall a standard estimate for the trace of a $H^1$ function on $\partial \Omega$ in the following lemma.
\begin{lemma}{\label{std_bd_trace}}
Assume that $v\in H^1(\Omega)$, then there exists a constant $C$ such that
\begin{equation}
\Vert v \Vert_{L^2(\partial\Omega)}^2 \leq C\Vert v \Vert_{L^2(\Omega)}(\Vert v \Vert_{L^2(\Omega)}+\vert v \vert_{1,\Omega}).
\end{equation}
\end{lemma}
\begin{proof}
This result is given in (4.37) in \cite{2013LarsonBengzon}.
\end{proof}

For each function $v\in PH^2(\Omega)\oplus S_h(\Omega)$, we let $J_hv$ be its interpolation in the
standard continuous (i.e., $H^1$) linear or bilinear finite element space defined on the same mesh $\mathcal{T}_h$ such that
\begin{equation}
\label{standard_interp}
J_hv|_T = J_{h, T}v, \hspace{0.1in} \text{with} \hspace{0.1in}
J_{h, T}v(X) = \sum_{i \in \mathcal{I}_{T}} v(A_i)\psi_{i, T}(X), ~~\forall X \in T,~~\forall T \in \mathcal{T}_h.
\end{equation}
Upper bounds of $J_hv$ are given in the following lemma.
\begin{lemma}{\label{est_interpolation_bnd}}
There exists a constant $C$ such that the following hold for
all $v\in PH^2(\Omega)\oplus S_h(\Omega)$:
\begin{align}
\Vert J_h v\Vert_{L^2(\Omega)} &\leq  \Vert v \Vert_{L^2(\Omega)}+Ch\vert v \vert_{1,\Omega}, \label{est_interpolation_bnd_1}\\
\vert J_hv\vert_{1,\Omega} &\leq C\vert v \vert_{1,\Omega}. \label{est_interpolation_bnd_2}
\end{align}
\end{lemma}
\begin{proof}
Let $v \in PH^2(\Omega)\oplus S_h(\Omega)$. Then $v|_T \in H^2(T)$ on $T\in\mathcal{T}_h^n$ and $v|_T \in PH^2(T)\oplus S_h(T)$ on $T\in\mathcal{T}_h^i$.
Denote $Y_i(t,X)=tA_i+(1-t)X$, $t\in [0,1]$. By the first order Taylor expansion, we have
 \begin{equation}
 v(A_i)=v(X)+\int_0^1 \nabla v (Y_i(t,X)) \cdot (A_i-X) dt. \label{eq:est_interpolation_bnd_1}
 \end{equation}
Using \eqref{eq:est_interpolation_bnd_1} and the partition of unity of the linear and bilinear finite element
shape functions on $T \in \mathcal{T}_h$, we have
\begin{equation}
 J_hv(X) = J_{h, T}v(X) = v(X)+\sum_{i\in\mathcal{I}_T}\left(\int_0^1 \nabla v (Y_i(t,X)) \cdot (A_i-X)dt\right)\psi_i(X),~~
 \forall X \in T. \label{eq:est_interpolation_bnd_2}
 \end{equation}
Since there exists a constant $C$ such that $\Vert\psi_i\Vert_{L^{\infty}(T)}\leq C$ and $\Vert A_i-X\Vert\leq Ch$,
from \eqref{eq:est_interpolation_bnd_2}, we have
\begin{equation}{\label{est_jh_l2}}
\begin{split}
\Vert J_hv\Vert_{L^2(T)}\leq &\Vert v \Vert_{L^2(T)}+C \Big (\int_T \big(\sum_{i\in\mathcal{I}_T}\int_0^1 \nabla v (Y_i(t,X)) \cdot (A_i-X)dt\big)^2dX\Big)^{1/2}\\
\leq&\Vert v \Vert_{L^2(T)}+Ch\int_0^1\big(\sum_{i\in\mathcal{I}_T}\int_T\Vert \nabla v (Y_i(t,X))\Vert^2 dX\big)^{\frac{1}{2}}dt\\
\leq &\Vert v \Vert_{L^2(T)}+Ch\vert v \vert_{1,T}.
\end{split}
\end{equation}
Similarly, by $\Vert \nabla \psi_{i,T}\Vert_{L^\infty(T)}\leq Ch^{-1}$, we have
\begin{equation}{\label{est_jh_h1}}
\begin{split}
\Vert \nabla J_hv\Vert_{L^2(T)}=&\Vert\sum_{i\in \mathcal{I}_T} v(A_i) \nabla \psi_{i,T}(X)\Vert_{L^2(T)}\\
=& \left(\int_T(\sum_{i\in\mathcal{I}_T}\int_0^1 \nabla v(Y_i(t,X))(A_i-X)dt\nabla \psi_{i,T}(X))^2dX\right)^{1/2}  \\
\leq & Ch^{-1} Ch\vert v \vert_{1,T}\\
\leq &C\vert v \vert_{1,T}.
\end{split}
\end{equation}
Then, summing \eqref{est_jh_l2} and \eqref{est_jh_h1} over all elements $T\in\mathcal{T}_h$ leads to estimates \eqref{est_interpolation_bnd_1} and
\eqref{est_interpolation_bnd_2}, respectively.
\end{proof}
Since the IFE space $S_h(\Omega)$ is not a subspace of $H^1(\Omega)$ in general \cite{2008HeLinLin,2004LiLinLinRogers}, the trace inequality
cannot be applied to
functions in $PH^2(\Omega)\oplus S_h(\Omega)$, for which, nevertheless, we can derive a similar trace inequality as follows.
\begin{thm}{\label{bound_trace_ineq}}
There exists a constant $C$ such that for every $v \in PH^2(\Omega)\oplus S_h(\Omega)$ the following inequality holds:
\begin{equation}
\Vert v\Vert_{L^2(\partial \Omega)}^2\leq C \big(\Vert v \Vert_{L^2(\Omega)}+h\vert v \vert_{1,\Omega}\big)\Vert v\Vert_{1,\Omega}. \label{eq:bound_trace_ineq}
\end{equation}
\end{thm}
\begin{proof}
Let $v$ be a function in $PH^2(\Omega)\oplus S_h(\Omega)$ and let $J_hv$ be its standard finite element interpolation described by \eqref{standard_interp}, and we have
\begin{equation}{\label{eq_bci_1}}
\Vert v \Vert_{L^2({\partial\Omega})}^2\leq 2 \left(\Vert v-J_hv\Vert_{L^2(\partial \Omega)}^2+\Vert J_hv\Vert_{L^2(\partial \Omega)}^2\right) .
\end{equation}
We estimate the second term on the right hand side of \eqref{eq_bci_1} first. Since $J_h v$ is in $H^1(\Omega)$, by Lemma \ref{std_bd_trace} and Lemma \ref{est_interpolation_bnd}, we have
\begin{equation}{\label{eq_bci_2}}
\begin{split}
\Vert J_hv\Vert_{L^2(\partial \Omega)}^2\leq& C\Vert J_hv\Vert_{L^2(\Omega)}(\Vert J_hv\Vert_{L^2(\Omega)}+\vert J_hv\vert_{1,\Omega})\\
\leq & C(\Vert v \Vert_{L^2(\Omega)}+Ch\vert v\vert_{1,\Omega})(\Vert v \Vert_{L^2(\Omega)}+Ch\vert v \vert_{1,\Omega}+\vert v \vert_{1,\Omega})\\
\leq & C(\Vert v \Vert_{L^2(\Omega)}+h\vert v \vert_{1,\Omega})\Vert v \Vert_{1,\Omega}.
\end{split}
\end{equation}
For the first term on the right hand side of \eqref{eq_bci_1}, we note that $v\in H^2(T)$ on $T\in \mathcal{T}_h^b$ because of the assumption that the interface $\Gamma$ does not touch boundary elements when $h$ is small enough. Then, using the standard trace inequality on $T\in \mathcal{T}_h^b$ and the approximation capability of finite element space, we have
\begin{equation}{\label{eq_bci_3}}
\begin{split}
\Vert v-J_hv\Vert_{L^2(\partial \Omega)}^2\leq &\sum_{T\in\mathcal{T}_h^b} \Vert v-J_hv\Vert_{L^2(\partial T)}^2\\
\leq& Ch^{-1}\sum_{T\in\mathcal{T}_h^b} (\Vert v-J_hv\Vert_{L^2(T)}^2+h^2\Vert \nabla (v-J_hv)\Vert_{L^2(T)}^2)\\
\leq& Ch^{-1} \sum_{T\in\mathcal{T}_h^b} (Ch^2 | v |_{1,T}^2+h^2\cdot C| v |_{1,T})^2\\
\leq & Ch \sum_{T\in\mathcal{T}_h^b} | v |_{1,T}^2\\
\leq & Ch | v |_{1,\Omega}^2.
\end{split}
\end{equation}
Finally, the inequality \eqref{eq:bound_trace_ineq} follows from applying \eqref{eq_bci_2} and \eqref{eq_bci_3} to \eqref{eq_bci_1}.
\end{proof}

For each function $u\in PH^2(\Omega)$, we recall that its interpolation in the IFE space $S_h(\Omega)$ is as \cite{2009HeTHESIS,2008HeLinLin,2004LiLinLinRogers}
\begin{eqnarray}
\label{ife_interpolation}
I_hu|_T = I_{h, T}u, \hspace{0.1in} \text{with} \hspace{0.1in}
\begin{cases}
      & I_{h, T}u(X) = \sum_{i \in \mathcal{I}_T} u(A_i)\phi_{i, T}(X), ~~\forall X \in T,~~\forall T \in \mathcal{T}^i_h , \\
      &  I_{h, T}u(X) = \sum_{i \in \mathcal{I}_T} u(A_i)\psi_{i, T}(X), ~~\forall X \in T,~~\forall T \in \mathcal{T}^n_h.
\end{cases}
\end{eqnarray}
The following theorem provides a description about the approximation capability of the IFE spaces in terms of the energy norm $\vertiii{.}_{\mathcal{H}}$
\begin{thm}{\label{interp_error}}
There exists a constant $C$ such that the following estimate holds for every $u\in PH^2(\Omega)$:
\begin{equation}{\label{inter_error_energy_2}}
\vertiii{I_hu-u}_{\mathcal{H}}\leq Ch\Vert u \Vert_{2,\Omega}, ~~\forall u\in PH^2(\Omega),
\end{equation}
provided that $k h\leq C_0$ for some constant $C_0$.
\end{thm}
\begin{proof} By Theorem 3.14 in \cite{2008HeLinLin}, Theorem 3.7 in \cite{2004LiLinLinRogers}, and Theorem 4.2 in \cite{2018GuoLinZhuang},  it follows
\begin{equation*}
\begin{split} 
\interleave{I_hu-u}\interleave_{\mathcal{H}}^2&=\interleave{I_hu-u}\interleave_{{h}}^2+k^2\Vert I_hu-u\Vert_{L^2(\Omega)},\\
&\leq Ch^2\Vert u \Vert_{2,\Omega}^2+Ck^2 h^4\Vert u \Vert_{2,\Omega}^2 \leq Ch^2\Vert u \Vert_{2,\Omega}^2,
\end{split}
\end{equation*}
which proves \eqref{inter_error_energy_2}.
\end{proof}

We now proceed to the error estimation for the symmetric PPIFE methods described by \eqref{weak_form_ppife}, and we will follow
Schatz's argument \cite{1974Schatz}. We start from the G{\aa}rding's inequality for $b_h(.,.)$ in the following lemma.

\begin{lemma}{\label{pseudo_coercivity}}
There exist constants $C_1$ and $C_2$ such that
the following inequality holds for $\sigma_e^0$ sufficiently large
\begin{equation}{\label{coercivity}}
\vert b_h(v,v) \vert \geq C_1 \interleave v \interleave_\mathcal{H}^2-C_2k^2\Vert v \Vert_{L^2(\Omega)}, \quad \forall v\in S_h(\Omega).
\end{equation}
\end{lemma}
\begin{proof}
First of all, we note that
\begin{equation}
\begin{split}
\vert b_h(v,v)\vert &\geq \frac{1}{\sqrt{2}} \big({\rm{Re}}(b_h(v,v))+{\rm{Im}}(b_h(v,v))\big)\\
&=\frac{1}{\sqrt{2}} \left({\rm{Re}}\big(a_h(v,v)\big)+{\rm{Im}}\big(a_h(v,v)\big) + k\Vert v \Vert_{L^2(\partial \Omega)}^2-k^2\Vert v \Vert_{L^2(\Omega)}^2 \right).
\end{split} \label{eq:pseudo_coercivity_1}
\end{equation}
Next, we introduce the bilinear form ${\tilde{a}_h(.,.)}$ : $V_h(\Omega)\times V_h(\Omega)\rightarrow\mathbb{C}$ such that
\begin{equation*}
a_h(u,v)={\tilde{a}_h(u,v)}+i\sum_{e\in\mathring{\mathcal{E}}_h^i} \frac{\sigma^0_e}{|e|} \int_e [u]_e\,[\overline{v}]_e ds,~~\forall v\in V_h(\Omega).
\end{equation*}
For each $v\in S_h(\Omega)$, we let $v=v_1+iv_2$ with $v_1=\mathrm{Re}(v) \in \tilde{S}_h(\Omega)$ and $v_2=\mathrm{Im}(v) \in \tilde{S}_h(\Omega)$. Since ${{a}_h(.,.)}$ and ${\tilde{a}_h(.,.)}$ are both bilinear and symmetric, we have
\begin{equation*}
a_h(v,v)=a_h(v_1,v_1)+a_h(v_2,v_2).
\end{equation*}
It follows that
\begin{equation}
\mathrm{Re}\big(a_h(v,v)\big)+\mathrm{Im}\big(a_h(v,v)\big) = {\tilde{a}_h(v_1,v_1)}+{\tilde{a}_h(v_2,v_2)}+\sum_{e\in\mathring{\mathcal{E}}_h^i} \frac{\sigma^0_e}{|e|} \int_e [v_1]_e\,[\overline{v_1}]_e ds + \sum_{e\in\mathring{\mathcal{E}}_h^i}\frac{\sigma^0_e}{|e|} \int_e [v_2]_e\,[\overline{v_2}]_e ds. \label{eq:pseudo_coercivity_2}
\end{equation}
Because $v_1,v_2\in \tilde{S}_h(\Omega)$, we can apply Theorem 4.3 in \cite{2018GuoLinZhuang} to \eqref{eq:pseudo_coercivity_2} so that
there exists a constant $\kappa>0$ such that
\begin{equation}
\mathrm{Re}\big(a_h(v,v)\big) + \mathrm{Im}\big(a_h(v,v)) \geq \kappa (\vertiii{ v_1}_h^2+\vertiii{ v_2}_h^2)=\kappa\vertiii{v}_h^2.
\label{eq:pseudo_coercivity_3}
\end{equation}
Therefore, applying \eqref{eq:pseudo_coercivity_3} to \eqref{eq:pseudo_coercivity_1} we have
\begin{equation*}
\begin{split}
\vert b_h(v,v)\vert &\geq \frac{1}{\sqrt{2}} \left(\kappa\vertiii{v}_h^2 + \kappa k^2 \Vert v\Vert_{L^2(\Omega)}^2-k^2(1+\kappa)\Vert v \Vert_{L^2(\Omega)}^2\right),\\
&\geq \frac{1}{\sqrt{2}} \left(\kappa\vertiii{v}_\mathcal{H}^2 - k^2(1+\kappa)\Vert v \Vert_{L^2(\Omega)}^2\right),
\end{split}
\end{equation*}
which proves \eqref{coercivity}.
\end{proof}

The following lemma is about the continuity of the bilinear form $b_h(\cdot, \cdot)$.
\begin{lemma}{\label{boundness}}
There exists a constant $C$ such that for every $y,v\in PH^2(\Omega)\oplus S_h(\Omega)$ the following inequality holds
\begin{equation}{\label{bilinear_bound}}
\vert b_h(y,v)\vert \leq C\vertiii{y}_{\mathcal{H}} \vertiii{v}_{\mathcal{H}},
\end{equation}
provided that $k h\leq C_0$ for some constant $C_0$.
\end{lemma}
\begin{proof}
By the same arguments used for proving Theorem 4.4 in \cite{2018GuoLinZhuang}, we can show that there exists a constant $C$ such that
\begin{equation*}
\vert a_h(y,v) \vert \leq {C}\vertiii{y}_{h} \vertiii{v}_{h}.
\end{equation*}
Since $\Vert y\Vert_{\mathcal{H}}\geq k \Vert y\Vert_{L^2(\Omega)}$, then
\begin{equation}{\label{bnd_eqn}}
\begin{split}
\vert b_h(y,v)\vert & \leq {C}\vertiii{y}_{h} \vertiii{v}_{h}
+k^2 \Vert y \Vert_{L^2(\Omega)}\Vert v \Vert_{L^2(\Omega)}+Ck\Vert y\Vert_{L^2(\partial\Omega)}\Vert v \Vert_{L^2(\partial\Omega)}\\
& \leq {C}\vertiii{y}_{\mathcal{H}} \vertiii{v}_{\mathcal{H}}+\vertiii{y}_{\mathcal{H}} \vertiii{v}_{\mathcal{H}}+Ck\Vert y\Vert_{L^2(\partial\Omega)}\Vert v \Vert_{L^2(\partial\Omega)}.
\end{split}
\end{equation}
For the third term on the right hand side of \eqref{bnd_eqn}, applying Theorem \ref{bound_trace_ineq}, we have
\begin{equation}{\label{bnd_eqn2}}
\begin{split}
k^2\Vert y\Vert_{L^2(\partial\Omega)}^2\Vert v \Vert_{L^2(\partial\Omega)}^2\leq &Ck^2(\Vert y \Vert_{L^2(\Omega)}+h\vert y \vert_{1,\Omega})(\Vert v \Vert_{L^2(\Omega)}+h\vert v \vert_{1,\Omega})\Vert y\Vert_{1,\Omega}\Vert v\Vert_{1,\Omega}\\
= & C(k\Vert y \Vert_{L^2(\Omega)}+k h\vert y \vert_{1,\Omega})(k \Vert v \Vert_{L^2(\Omega)}+k h\vert v \vert_{1,\Omega})\Vert y\Vert_{1,\Omega}\Vert v\Vert_{1,\Omega}\\
\leq & C\interleave y \interleave_{\mathcal{H}}^2\interleave v \interleave_{\mathcal{H}}^2.
\end{split}
\end{equation}
Thus, applying \eqref{bnd_eqn2} to \eqref{bnd_eqn} leads to \eqref{bilinear_bound}.
\end{proof}

Following Schatz's argument \cite{1974Schatz}, we now derive a posteriori error estimate for the symmetric PPIFE solution in the following lemma.
\begin{lemma}{\label{post_eriori_est}}
Let $u\in PH^2(\Omega)$ be the exact solution to the problem \eqref{model}, and let $u_h$ be the solution produced by symmetric PPIFE method  ~\eqref{weak_form_ppife} with $\sigma_e^0$ large enough, then there exists a constant $C$  such that
\begin{equation}
\Vert u-u_h \Vert_{L^2(\Omega)}\leq C(k + 1/k) h\interleave u-u_h\interleave_{\mathcal{H}}, \label{eq:post_eriori_est}
\end{equation}
provided that $k h \leq C_0$ for some constant $C_0$.
\end{lemma}
\begin{proof}
We define an auxiliary function $z\in PH^2(\Omega)$ as the solution to the problem \eqref{model}, with $f$ replaced by $e = u-u_h$ and $g$ replaced by the zero function. In the weak form (\ref{weak_form_pp}) for $z$, choosing $v=e$ as the test function yields
\begin{equation*}
\Vert e \Vert_{L^2(\Omega)}^2=b_h(z,e).
\end{equation*}
Let $I_hz$ be the interpolent of $z$ in IFE space $S_h(\Omega)$ defined by \eqref{ife_interpolation}, it follows
\begin{equation*}
\begin{split}
b_h(I_hz, e)&=b_h(I_hz,u)-b_h(I_hz,u_h)\\
&=(f,I_hz)_{\Omega} + (g,I_hz)_{\partial \Omega} - (f,I_hz)_{\Omega} - (g,I_hz)_{\partial \Omega}\\
&=0.
\end{split}
\end{equation*}
Thus $b_h(z,e)=b_h(z-I_hz, e)$.
Therefore, by Lemma \ref{boundness}, Theorem \ref{interp_error} and Assumpsion \ref{regularity_assumption}, we have
\begin{equation*}
\begin{split}
\Vert e\Vert_{L^2(\Omega)}^2=& b_h(z-I_hz, e)\\
&\leq C\vertiii{z-I_hz}_{\mathcal{H}} \vertiii{e}_{\mathcal{H}} \\
&\leq Ch\Vert z \Vert_{2,\Omega}\vertiii{e}_{\mathcal{H}}\\
&\leq C(k+1/k) h \Vert e \Vert_{L^2(\Omega)}\vertiii{e}_{\mathcal{H}},
\end{split}
\end{equation*}
which proves \eqref{eq:post_eriori_est}.
\end{proof}

Now, we are ready to derive the optimal error bounds in both the energy norm and $L^2$ norm for the symmetric PPIFE methods described by \eqref{weak_form_ppife}.

\begin{thm}{\label{error_H1_est}}
Under the conditions of Lemma~\ref{post_eriori_est}, there exists a constant $C$  such that
\begin{align}
\vertiii{{u-u_h}}_\mathcal{H}\leq Ch\Vert u \Vert_{2,\Omega} \label{soln_bnd_2},
\end{align}
provided that $(k^2 + 1) h$ is sufficiently small.
\end{thm}
\begin{proof}
First, we assume that $(k^2 + 1) h$ is sufficiently small such that $kh \leq C_0$ for some constant $C_0$.
Denote $e=u-u_h$, $e_h=u_h-I_hu$, then by Lemma \ref{pseudo_coercivity} and Lemma \ref{boundness}, we have
\begin{equation*}
C_1\interleave{e_h}\interleave_\mathcal{H}^2-C_2k^2 \Vert e_h\Vert_{L^2(\Omega)}^2 \leq  b_h(e_h,e_h)=b_h(u-I_hu, e_h)\leq C\interleave u-I_hu\interleave_\mathcal{\mathcal{H}} \interleave e_h\interleave_\mathcal{H}.
\end{equation*}
By the fact that $\vertiii{e_h}_{\mathcal{H}}\geq k \Vert e_h\Vert_{L^2(\Omega)}$, we then have
\begin{equation*}
\begin{split}
\vertiii{e_h}_\mathcal{H}^2 &\leq C\vertiii{u-I_hu}_\mathcal{H}\vertiii{e_h}_{\mathcal{H}}+Ck^2 \Vert e_h\Vert_{L^2(\Omega)}^2 \\
&\leq C \vertiii{u-I_h u}_\mathcal{H}\vertiii{e_h}_{\mathcal{H}}+Ck\vertiii{e_h}_{\mathcal{H}}\Vert e_h\Vert_{L^2(\Omega)}.
\end{split}
\end{equation*}
Therefore, using Theorem \ref{interp_error}, we have
\begin{equation*}
\begin{split}
\interleave e_h\interleave_\mathcal{H}&\leq C\interleave u-I_hu\interleave_{\mathcal{H}}+Ck \Vert e_h\Vert_{L^2(\Omega)}, \\
&\leq Ch\Vert u \Vert_{2,\Omega}+Ck(\Vert e \Vert_{L^2(\Omega)}+\Vert u-I_hu\Vert_{L^2(\Omega)}).
\end{split}
\end{equation*}
Furthermore, by Lemma \ref{post_eriori_est} and the approximation capability of IFE spaces \cite{2008HeLinLin,2004LiLinLinRogers},
we have
\begin{equation}{\label{est_eh}}
\begin{split}
\interleave e_h\interleave_\mathcal{H} \leq & Ch\Vert u \Vert_{2,\Omega}+Ck\big(C(k + 1/k) h\vertiii{e}_{\mathcal{H}}+Ch^2\Vert u \Vert_{2,\Omega}\big)\\
\leq & Ch\Vert u \Vert_{2,\Omega} + C k(k + 1/k) h\vertiii{e}_\mathcal{H} + Ck h^2\Vert u \Vert_{2,\Omega},\\
\leq & C h\Vert u \Vert_{2,\Omega}+C (k^2 + 1) h\vertiii{e}_\mathcal{H}.
\end{split}
\end{equation}
Hence, by \eqref{est_eh}, we have
\begin{equation*}
\begin{split}
\vertiii{e}_{\mathcal{H}}&\leq \vertiii{u-I_hu}_{\mathcal{H}} + \vertiii{e_h}_{\mathcal{H}}\\
&\leq \vertiii{u-I_hu}_{\mathcal{H}} + Ch\Vert u \Vert_{2,\Omega} + C(k^2 + 1) h\vertiii{e}_{\mathcal{H}} .
\end{split}
\end{equation*}
Then, by Theorem \ref{interp_error} again, the inequality above becomes
\begin{equation*}
\big(1-C(k^2 + 1) h\big)\interleave e \interleave_{\mathcal{H}}\leq Ch\Vert u \Vert_{2,\Omega},
\end{equation*}
which proves \eqref{soln_bnd_2} provided that $(k^2 + 1) h$ is sufficiently small.
\end{proof}
\begin{rem}
Resort to the idea in \cite{1974Schatz}, if $u_h$ is a PPIFE solution corresponding to $u=0$, then from Theorem \ref{error_H1_est} it follows that $u_h=0$ provided that $h$ is sufficiently small guaranteeing $(k^2 + 1)h$ is sufficiently small. This implies that the linear system to solve $u_h$ induced from the symmetric PPIFE scheme \eqref{weak_form_ppife} is nonsingular; therefore, the PPIFE solution $u_h$ defined by \eqref{weak_form_ppife} exists and is unique.
\end{rem}
\begin{thm}{\label{error_L2_est}}
Under the conditions of  Theorem \ref{error_H1_est}, there exists a constant $C$, such that
\begin{equation}
\Vert u-u_h\Vert_{L^2(\Omega)}\leq   C(k + 1/k) h^2\Vert u \Vert_{2,\Omega}. \label{eq:error_L2_est}
\end{equation}
\end{thm}
\begin{proof}
The estimate \eqref{eq:error_L2_est} follows directly from Lemma~\ref{post_eriori_est} and Theorem~\ref{error_H1_est}.
\end{proof}
\section{A Numerical Example}
In this section, we present a numerical example to validate the error estimates in Theorems \ref{error_H1_est} and \ref{error_L2_est}.
We note that \cite{2019LinLinZhuang} provides quite a few numerical examples to illustrate convergence features of the PPIFE
methods developed there for solving the Helmholtz interface problems. However, the exact solutions in the examples presented in
\cite{2019LinLinZhuang} have a regularity better than piecewise $H^r$ with $r >2$. Hence, it is interesting to see how the PPIFE solution
converges when the exact solution only has piecewise $H^2$ regularity.

Specifically, let the domain $\Omega = (-1,1)\times (-1,1)$ be separated by the circular interface $\Gamma:x^2+y^2-r_0^2=0$, $r_0=\pi/6.28$ into two subdomains
$$
\Omega^-=\{ (x,y): x^2+y^2 < r_0^2 \}, ~~~ \Omega^+=\Omega\backslash\overline{\Omega^-}.
$$
We generate a Cartesian triangular mesh $\mathcal{T}_h$ of $\Omega$ by partitioning $\Omega$ into $N\times N$ congruent squares so that $h=2/N$, and then partitioning each square into two congruent triangles by its diagonal line. We let functions $f$ and $g$ in the interface problem \eqref{model} be generated with the following exact solution:
\begin{equation}
\label{example_u}
u(x,y) =\left\{
\begin{aligned}
&\frac{2+i}{\beta^-}r^{\alpha},  &\; (x,y)\in \Omega^-,\\
&\frac{2+i}{\beta^+}r^{\alpha}+\left( \frac{2+i}{\beta^-} -\frac{2+i}{\beta^+} \right)r_0^{\alpha}, & \; (x,y) \in \Omega^+,
\end{aligned}
\right.
\end{equation}
where $\alpha = 1.5, r = \sqrt{x^2 + y^2}$. We choose $\sigma_e^0=30\max\{\beta^-,\beta^+\}$ for the parameter required in \eqref{eq:bilinear_form_pp}.  It can be verified that, $u \in PH^{2}(\Omega) \backslash PH^{3}(\Omega)$. Table \ref{table:CN_SPPIFE_1_20} presents errors of the symmetric PPIFE solutions $u_h$ generated on a sequence of uniform triangular meshes $\mathcal{T}_h$ of $\Omega$ in a certain configuration of $k$, $\beta^-$, $\beta^+$.  The results demonstrate that, for fixed $k$, the symmetric PPIFE solutions converge optimally in both semi-$H^1$ and $L^2$ norms to the exact solution $u\in PH^{2}(\Omega) \backslash PH^{3}(\Omega)$, and this validates the theoretical results established in Theorem \ref{error_H1_est}
and Theorem \ref{error_L2_est} in the previous section.

\begin{table}[H]
\centering
\begin{tabular}{ |c |c c| c c|}
\hline
$N$    & $\|u - u_h\|_{0,\Omega}$  & rate & $|u - u_h|_{1,\Omega}$ & rate   \\ \hline
  10   &  3.6019e-02   & NA       &  5.2313e-01   &NA\\
  20   &  1.6412e-02   & 1.1340    & 2.5292e-01  & 1.0485    \\
  40   &  6.6539e-03   & 1.3025    & 1.1802e-01  & 1.0997    \\
  80   &  1.3425e-03   & 2.3092    & 5.1500e-02  & 1.1964    \\
 160  &  2.7744e-04   & 2.2747    & 2.4983e-02  & 1.0436    \\
 320  &  7.7328e-05   & 1.8431    & 1.2427e-02  & 1.0075    \\
 640  &  1.9455e-05   & 1.9909    & 6.1961e-03  & 1.0040    \\
1280 &  4.7698e-06   & 2.0281    & 3.0947e-03  & 1.0015    \\\hline
\end{tabular}
\caption{Errors of the PPIFE solution, $k=10$, $\beta^-=1$, $\beta^+=10$.}
\label{table:CN_SPPIFE_1_20}
\end{table}
\section*{Acknowledgements}
Yanping Lin was partially supported by GRF B-Q56D of HKSAR and Polyu G-UA7V.

\bibliographystyle{plain}
\bibliography{QZBib}

\end{document}